\documentclass[11pt]{amsart}
\usepackage{amssymb}
\usepackage[headings]{fullpage}
\usepackage{amsfonts}
\usepackage{paralist}
\usepackage{color}

\usepackage[colorlinks,pagebackref=true,pdftex]{hyperref}

\makeatletter
\@addtoreset{equation}{section}
\def\theequation{\thesection.\@arabic \c@equation}

\def\theenumi{\@roman\c@enumi}
\def\@citecolor{blue}
\def\@linkcolor{blue}
\def\@urlcolor{blue}

\makeatother

\newtheorem{theorem}[equation]{Theorem}

\newtheorem{corollary}[equation]{Corollary}

\newtheorem{claim*}{Claim}

\theoremstyle{definition}
\newtheorem{remark}[equation]{Remark}

\newtheorem{eg}[equation]{Example}
\newenvironment{example}[1][]{\begin{eg}[#1] \pushQED{\qed}}{\popQED
\end{eg}}
\newtheorem{definition}[equation]{Definition}

\newtheorem{notn}[equation]{Notation}

\newsavebox{\upperboundtheorem}

\title{A note on the subadditivity of Syzygies}

\author[S.~El~Khoury]{Sabine El Khoury} 
\address{Department of Mathematics, American University of Beirut, Beirut,
Lebanon.}
\email{se24@aub.edu.lb} 

\author[H.~Srinivasan]{Hema Srinivasan}
\address{Department of Mathematics, University of Missouri, Columbia,
Missouri-65211, USA.}
\email{srinivasanh@missouri.edu}

\begin{document}

 \begin{abstract}
Let $R=S/I$ be a graded algebra  with $t_i$ and $T_i$ being the minimal and maximal shifts in the minimal graded free resolution of $R$ at degree $i$.  We prove that  $t_n\leq t_1+T_{n-1}$ for all $n$. As a consequence, we show that for Gorenstein algebras of codimension $h$, the subadditivity of maximal shifts $T_n$ in  the minimal graded free resolution holds for  $n \ge h-1$, i.e we show that $T_n \leq T_a+T_{n-a}$ for $n\ge h-1$. \end{abstract}

\maketitle

\section{Introduction}  
Let $R= S /I$ be a standard graded algebra, with $S$ a polynomial ring over a field $k$.    Denote  by $(\mathbb F, \partial) $ a minimal graded free resolution of $R$  over $S$ with  $\mathbb F_a= \oplus_j S(-j)^{\beta_{aj}}$.  For each $a$,  denote by $T_a(\mathbb F)$ and $t_a(\mathbb F)$ the maximal and minimal shifts in the resolution $\mathbb F$.   In other words, 
$$t_a (\mathbb F)= min \left\{ j: \beta_{aj} \neq 0 \right\} $$
$$T_a(\mathbb F)= max \left\{ j: \beta_{aj} \neq 0 \right\}.$$
$\mathbb F$ is said to satisfy the {\it subadditivity} condition for maximal shifts if $T_{a+b}(\mathbb F) \leq T_a(\mathbb F)+T_b(\mathbb F)$, for all $a$ and $b$.   
If $\mathbb F$ is the minimal graded free resolution of $R$ over $S$, then we write $t_a$ and $T_a$  for $t_a(\mathbb F)$ and   $T_a(\mathbb F)$.   

General bounds for $T_n$ in terms of $T_1$ were established for all $n$, by  Bayer-Mumford and Caviglia-Sbarra in \cite {BM} and \cite{CaS}. They prove that $T_n \leq (2T_1)^{2^{m-2}} -1+n$ where $m=$rank$_kS_1$. Subadditivity for maximal shifts in minimal graded free resolutions has been studied by many authors \cite{ACI2}, \cite{EHU}, \cite{HS}, \cite {Mc}.  Subadditivity for maximal shifts has been established for the highest non vanishing Tor, $Tor_p( R, k) $. It was shown that $T_p\le T_1+T_{p-1}$ for all graded algebras where $p = $projdim$R$ \cite[Corollary 3]{HS}, and that $T_p\le T_a+T_{p-a}$ when $R$ is of depth zero and of dimension $\le 1$ \cite[Corollary 4.1]{EHU}.  In \cite{ACI2}, Avramov, Conca and Iyengar consider the situation when $R=S/I$ is Koszul and show that $T_{a+1}(I) \leq T_a+T_1 = T_a+2$ for  $a \leq$ height $(I)$. Further, they prove that if $S$ is Cohen-Macaulay, $reg^R_{h+1}(k) = 0$ and ${a+b \choose a}$ is invertible in $k$, then  $T_{a+b} \leq T_a+T_b+ 1$.  For monomial ideals, Herzog and Srinivasan show that $T_{a+1} \le T_a+T_1$, for all $a$  \cite[Corollary 4]{HS}. 

 It is known that the minimal graded free resolution of graded algebras may not satisfy the subadditivity for maximal shifts as shown by the counter example in \cite{ACI2}.  However, no counter examples are known for monomial ideals.  In fact, it can be easily seen that the subadditivity for maximal shifts does hold for algebras generated by stable monomial or square free strongly stable monomial ideals.   

In this paper, we consider the minimal and maximal degrees of the syzygies and show that for any $n$ we have $t_n \le t_1+T_{n-1}$.  As a consequence, we prove 
  for  Gorenstein algebras $S/I$  that $T_{a+b} \le T_a+T_b$ for  $a+b \ge s-1$ where $s =$ height $I =$ projdim $R$.   The question remains if the subadditivity for maximal shifts holds for all Gorenstein algebras.  It is particularly interesting since the simplest counter example to whether the subadditivity holds, is a codimension four  almost complete intersection linked to a codimension four Gorenstein algebra.  

\section{Syzygies of Gorenstein Algebras}

 Let $S= k[x_1,\ldots,x_d] $ and $I$ a homogeneous ideal in $S$. Let $\mathbb F$ be the minimal graded free resolution of $R=S/I$
 $$\mathbb F:  0 \to F_s \to \ldots \to F_2 \stackrel{\partial_2}{\longrightarrow} F_1 \stackrel{\partial_1}{\longrightarrow} S$$
 where $F_a= \oplus_j S(-j)^{\beta_{aj}}$, $t_a = min \left\{ j: \beta_{aj} \neq 0 \right\}$ and $T_a=  max \left\{ j: \beta_{aj} \neq 0 \right\}$. We prove a general result for the syzygies of homogeneous algebras. 
 
 \begin{theorem} \label{general}
 $R=S/I$ be a graded algebras with $t_i$ and $T_i$ being the minimal and maximal shifts in the minimal graded free $S$-resolution of $R$ at degree $i$, then $t_n\leq t_1+T_{n-1}$, for all $n$.
 \end{theorem}
 \begin{proof}  
 We show the theorem by induction on $n$, where $n$ is the $n^{th}$ step of the resolution.\\
For $n=1$, $T_0=0$ and hence $t_1 = t_1$.  We need to prove the theorem for $1< n\le s$,  where 
$$\mathbb F:  0 \to F_s \to \ldots \to F_2 \stackrel{\partial_2}{\longrightarrow} F_1 \stackrel{\partial_1}{\longrightarrow} S$$
is the graded resolution of $R$.   
 Let  $I=(g_1, \ldots, g_{\beta_1})$ where $\left\{g_1, \ldots, g_{\beta_1}\right\}$ is a set of minimal generators  of $I$ and $F_i = \bigoplus_{j=1}^{\beta_i} Sf_{ij} $ with $t_1=$deg$f_{11}\leq $deg$f_{12} \leq \cdots \leq$deg$f_{1\beta_1}=T_1$. Suppose  $ \partial_1(f_{1j})= g_j$ and $ \partial_{n}(f_{nt})= \sum_{i=1}^{ \beta_{n-1}} r_{ti} f_{(n-1)i}$ for all $n$. \\
  We see that   $Z(f_{11}, f_{1\beta_1}) = g_1f_{1 \beta_1}- g_{ \beta_1} f_{11}$ is a non zero second syzygy. Hence, there exists an element $f_{11}*f_{1\beta1} \in F_2$ of degree $t_1+T_1$. This implies that $t_2 \leq T_1 + t_1$, and the theorem is true for $n=2$.

Assume we have constructed by induction an element $f_{11}*f_{{n-2}i} \in F_{n-1}$ for all  $1 \leq i \leq \beta_{n-2}$ such that 
 $$\partial _{n-1} (f_{11}*f_{(n-2)i}) =  \partial_1(f_{11})f_{(n-2)i }- f_{11}*\partial _{n-2}(f_{(n-2)i}),$$  and $t_{n-1} \le t_1+T_{n-2}$ for $4 \leq n \leq s$. 
Let  
 $$f_{11}*f_{(n-2)i}= \sum_{j=1}^{\beta_{n-1}}s_{ij} f_{(n-1)j}.$$ We also have
  $ \partial_{n-1}(f_{(n-1)t})= \sum_{i=1}^{ \beta_{n-2}} r_{ti} f_{(n-2)i}$ for all $1 \leq t \le \beta_{n-1}$.
 Consider   $$Z(f_{11}, f_{n-1t}) = \partial_1(f_{11})f_{n-1t}- \sum_{i=1}^{\beta_{n-2}}\sum_{j=1}^{\beta_{n-1}}r_{ti}s_{ij}f_{(n-1)j}.$$ 
$Z(f_{11}, f_{n-1t})$ is a cycle in $F_{n-1}$, since
 \begin{align*}
 \partial_{n-1} (Z(f_{11}, f_{(n-1)t}))=& \partial_1(f_{11})\partial_{n-1}(f_{(n-1)t})- \sum_{i=1}^{\beta_{n-2}} r_{ti}\partial_{n-1}(\sum_{j=1}^{\beta_{n-1}}s_{ij}f_{(n-1)j})\\
  =& \partial_1(f_{11})\sum_{i=1}^{\beta_{n-2}}r_{ti}f_{(n-2)i}- \sum_{i=1}^{\beta_{n-2}} r_{ti}\partial_{n-1}(f_{11}*f_{(n-2)i})\\
 =&\sum_{i=1}^{\beta_{n-2}} \left[ \partial_1(f_{11})r_{ti}f_{(n-2)i}- \partial_1(f_{11})r_{ti}f_{(n-2)i}+ r_{ti}f_{11}*\partial _{n-2}(f_{(n-2)i}) \right]\\
=& \sum_{i=1}^{\beta_{n-2}}r_{ti} f_{11}*\partial _{n-2}(f_{(n-2)i}) \\
  =&  f_{11}*\partial _{n-2}( \sum_{i=1}^{\beta_{n-2}}r_{ti}f_{(n-2)i})\\
   =& f_{11}*\left(\partial _{n-2} \circ \partial_{n-1}(f_{(n-1)t}) \right) \\
    =& f_{11}*0\\
    =&0.\\
    \end{align*}
 Thus, for every $t$, $1\le t\le \beta_{n-1}$, there exists an element $f_{11}*f_{(n-1)t}$  $\in F_n$ of the same degree as deg$ Z(f_{11}, f_{(n-1)t}) $ that is mapped by $\partial_n$ onto $Z(f_{11}, f_{(n-1)t})$.  
  This means there is an element $f_{11}*f_{(n-1)t} \in F_n$ of degree $t_1+ deg f_{(n-1)t} \le t_1+T_{n-1}$ such that  $\partial_n( f_{11}*f_{(n-1)t})=Z(f_{11}, f_{(n-1)t}) $.   However, the cycles  $Z(f_{11}, f_{n-1t})$ can be zero for all $t$! For that, we show that there is at least one $t$  such that one of the cycles $Z(f_{11}, f_{n-1t})$  is not identically zero.  \\
Suppose that $Z(f_{11}, f_{n-1t}) =0$ for all $t$, then $\partial_1(f_{11})f_{n-1t}= \sum_{i=1}^{\beta_{n-2}}\sum_{j=1}^{\beta_{n-1}}r_{ti}s_{ij}f_{n-1j}  $. 
So for all $t$, we have
$$ \sum_{i=1}^{\beta_{n-2}} \sum_{j=1}^{\beta_{n-1}}r_{ti}s_{ij} = \begin{cases} 0&  j \neq t \\
\partial_1 (f_{11})&  j =t.
\end{cases}$$
 
By setting  $\bar{r}=\bar r_{n-1} = (r_{ti})_{\beta_{n-1} \times \beta_{n-2}}$ and $ \bar s =(s_{ij})_{\beta_{n-2} \times \beta_{n-1}}$,  we get $\bar r \bar s= \partial_1 (f_{11}) I$ and hence the rank $\bar r_{n-1} \bar s=\beta_{n-1}$.  By the exactness of the resolution, we have  rank $ \bar r_{n }+ $ rank  $ \bar r _{n-1}= \beta_{n-1}$ where $ \bar r_{n}$ is the matrix representing $\partial_{n}$.  Since $\partial _n \neq 0$, this implies that rank $\bar  r_{n-1} < \beta_{n-1}$ which is a contradiction.   

Hence, there exists a non zero cycle in $F_{n-1}$, and we get $t_n \leq t_1+ T_{n-1}$.
 \end{proof}
  
  The following remark is an easy consequence of the  duality of the minimal graded free resolution for Gorenstein algebras.  
 \begin{remark} If $R=S/I$ is a Gorenstein algebra, with height $I = h$,
then $T_{h} \leq T{_a}+T_{h-a}$.\\
Since $c= T_h=t_h$, then  by the duality of the minimal graded free resolution $\mathbb F$ we get $c-t_{h-a}=T_a$ for all $a=1, \cdots h-1$. This implies that $c=T_a+t_{h-a} \leq T_a+T_{h-a}$.\\
  \end{remark}
 
   \begin{theorem} For any graded Gorenstein algebra $S/I$ of codimension $h$, we have $T_{h-1}  \leq T_a + T_{h-1-a}$.
   Thus, $T_n\leq T_a+T_{n-a}$ for $n\geq h-1$.  
   \end{theorem}
   \begin{proof} Since $S/I$ is Gorenstein, then $T_{h-1} =  T_h-t_1$ and $T_{h-1-a} = T_h-t_{a+1}$.  So, 
$T_{h-a-1} = T_h-t_{a+1} \ge T_h -( t_1+T_a)$ by theorem \ref{general}.
So, $T_{h-a-1} \ge T_h- t_1-T_a = T_{h-1}-T_a$ and hence $T_{h-1} \le T_a+T_{h-a-1}$ as desired. 
\end{proof}

\begin{definition} A minimal  graded free resolution is said to be pure if, for all $a$, $ F_a$ is generated in one degree. Hence the minimal graded free resolution is pure if it has the following shape
 $$\mathbb F:  0 \to S(-j_s)^{\beta_{s}} \to \ldots \to S(-j_2)^{\beta_{2}} \stackrel{\partial_2}{\longrightarrow} S(-j_1)^{\beta_{1}} \stackrel{\partial_1}{\longrightarrow} S$$
\end{definition}
   \begin{corollary} For any graded algebra $S/I$ with a pure resolution, we get $T_n  \leq T_1 + T_{n-1}$. 
   \end{corollary}
   
   \begin{proof} When the resolution is pure we have $t_a=T_a$ for all $a$, and hence by theorem \ref{general} we get the result. 
   \end{proof}
   
   \begin{remark} As a consequence, if $S/I$ is a Gorenstein pure algebra of codimension $4$ or $5$, then its minimal graded free resolution satisfies the subadditivity condition.  
   \end{remark}
\begin{remark} As the following three examples show, the inequality can be tight even for codimension four Gorenstein algebras  with pure resolutions.
\end{remark}

\begin{example} Let $S = k[x,y,z,w]$ and $I = ( w^2, z^2, y^2-wz, wx^2, x^2z, x^2y+wyz, x^3+wxz+wyz)$ a height four Gorenstein ideal. We have $T_4=8$, $T_3=6$ $T_2=4$ and $T_1=3$. We get \\
$T_4=2T_2 \leq T_3+T_1$ and $T_3 \leq T_2+T_1$. 
\end{example}

\begin{example} Let $S=k[x, y, z, w]$ and $I=(w^3, w^2z, w^2y, wz^2, wyz-w^2x, wy^2, z^3+wxy, yz^2, xz^2, \\ y^2z, xyz-wx^2, x^2z, y^3-2wxz, xy^2, x^2y+w^2x, x^3+wxz-w^2x)$  a height four Gorenstein ideal generated in degree three and having a pure resolution.  Then the minimal graded free resolution of $S/I$ satisfies the subadditivity condition. We have $T_4=8$, $T_3=5$, $T_2=4$ and $T_1=3$. Then we get $T_4=2T_2= T_3+T_1$, $T_3 \leq T_2+T_1$ and  $T_2 \leq 2T_1$.
\end{example}

\begin{example} Let $S= k[x,y,z, w,r,s,t]$ and $I =( rt, tz, ty-st, tx, sw-s^2, sz, sy, sx, rw-st-t^2, rz, yr+tw, rx-s^2, w^2+st, wz, wy-s^2, wx, yz-z^2+r^2+rs+tw-st-t^2, xz-rs-tw+st+t^2, y^2, xy-z^2-r^2+rs+tw-st-t^2, x^2+rs+tw-st-t^2)$ a height seven Gorenstein ideal with $T_7=10$, $T_6=8$, $T_5=7$, $T_4=6$, $T_3=5$, $T_2=3$ and $ T_1=2$. Then we get\\
$T_7= T_6+T_1=T_5+T_2  \leq T_4+T_3$ and $T_6 \leq T_5+T_1=T_4+T_2 \leq 2T_3$.

\end{example}

\providecommand{\bysame}{\leavevmode\hbox to3em{\hrulefill}\thinspace}
\providecommand{\MRhref}[2]{%
  \href{http://www.ams.org/mathscinet-getitem?mr=#1}{#2}
}
\providecommand{\href}[2]{#2}

\end{document}